\documentclass{article}

\usepackage{amssymb} % to use \mathbb and \mathfrak
\usepackage{amsmath} % to use \varDelta and \ltimes
\usepackage{latexsym} % to use \Box, \Box is smaller than \mbox{\msa\char'003}
\usepackage{theorem} % for \documentclass{article} % to make Definitions in roman fonts 

\usepackage{epsfig}

\theorembodyfont{\itshape} % if we use \usepackage{theorem}

\newtheorem{theorem}{Theorem}[section]
\newtheorem{lemma}[theorem]{Lemma}
\newtheorem{corollary}[theorem]{Corollary}

\newtheorem{definition}[theorem]{Definition}
\newtheorem{example}{Example}[section]

\newenvironment{proof}{{\par\addvspace{0.1cm}\noindent \bf Proof. }}{\hfill$\Box$\par\medskip} % do not use it if we use \documentclass{amsproc} %\bfseries\itshape %{\footnotesize roof. }

\newenvironment{proofoftheorem}{{\par\noindent \bf Proof of Theorem}\hspace{0.1cm}}{\hfill$\Box$\par\medskip} % for \documentclass{amsbook}
\title{Conformal dual of a quadruplet of points}

\author{Jun O'Hara}

\begin{document}

\maketitle

\begin{abstract} 
%Let $\{P_1, P_2, P_3, P_4\}$ be a quadruplet of points in $S^3$. 
We define a ``dual'' of a quadruplet of points in $S^3$ in a conformal geometric way. 
We show that the dual of a dual quadruplet coincides with the original one. 
We also show that the cross ratio of the dual quadruplet is equal to the complex conjugate of that of the original one. 
\end{abstract}



\section{The non-cocircular case}
We start with the non-cocircular case. 
The cocircular case will be studied in section \ref{sec_cocircular}. 

\subsection{Definition of a dual quadruplet}
Let $P_1, P_2, P_3$, and $P_4$ be four points in $\mbox{\boldmath $S$}^3$ which are not cocircular, and $\varSigma$ a sphere through them. 
Let $\varGamma_{ijk}$ $(i\ne j\ne k\ne i)$ be an oriented circle through $P_i,P_j$, and $P_k$ whose orientation is given by the cyclic order of $P_i,P_j$, and $P_k$ along the circle. %(P_i,P_j,P_k)
\begin{definition}\label{def_circular_angle_bisector} \rm 
Suppose $\{i,j,k,l\}=\{1,2,3,4\}$. 
Let $\mathcal{S}_{ij}$ be the set of the oriented circles through $P_i$ and $P_j$ which intersect $\varGamma_{ijk}$ and $\varGamma_{ijl}$ in the same angle. 
Let $\varGamma_{ij}$ be the circle in $\mathcal{S}_{ij}$ which minimizes the intersection angles with $\varGamma_{ijk}$ and $\varGamma_{ijl}$. 
We call $\varGamma_{ij}$ the {\em circular angle bisector} of $\varGamma_{ijk}$ and $\varGamma_{ijl}$. 
\end{definition}
Let $\pi_i:\varSigma\to\varPi_i\cup\{\infty\}$ $(i=1,2,3,4)$ be a stereographic projection from $P_i$. 
It maps three circles thruogh $P_i$, $\varGamma_{ijk}$, $\varGamma_{ijl}$, and $\varGamma_{ikl}$, ($\{i,j,k,l\}=\{1,2,3,4\}$) to three lines which forms the triangle $\triangle \pi_i(P_{j})\pi_i(P_{k})\pi_i(P_{l})$. %in $\varPi_i$ 
Therefore, $\pi_i$ maps the circular angle bisector $\varGamma_{ij}$ to the angle-bisector of the angle $\angle \pi_i(P_{j})$. 
It follows that  $\varGamma_{ij}$ belongs to the sphere $\varSigma$. 

Three angle-bisectors $\pi_i(\varGamma_{ij}), \pi_i(\varGamma_{ik})$, and $\pi_i(\varGamma_{il})$ meet at the incenter of the triangle $\triangle \pi_i(P_{j})\pi_i(P_{k})\pi_i(P_{l})$. 
Let us denote it by $\widetilde P_i^{\prime}$ (Figure \ref{conf_dual1-incenter}). 
\begin{figure}[htbp]
\begin{center}
\includegraphics[width=.38\linewidth]{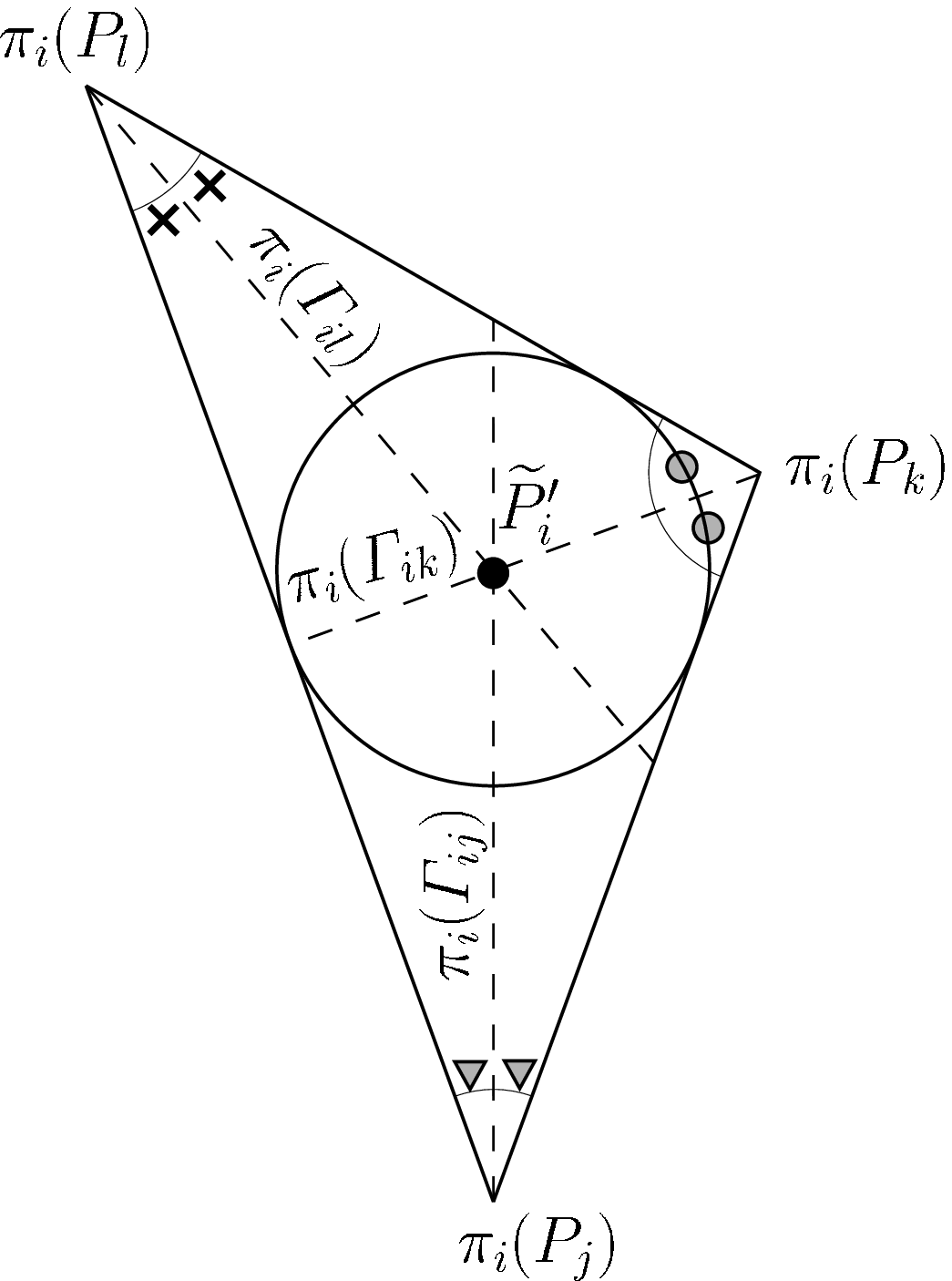}
\caption{The inceter of the triangle $\triangle\pi_i(P_{j})\pi_i(P_{k})\pi_i(P_{l})$}
\label{conf_dual1-incenter}
\end{center}
\end{figure}

\begin{definition}\label{def-dual} \rm 
As above, put $P_i^{\prime}=\pi_i{}^{-1}\left(\widetilde P_i^{\prime}\right)$. 
Namely, $P_i^{\prime}$ is given by 
\[\varGamma_{ij}\cap\varGamma_{ik}\cap\varGamma_{il}=\{P_i, P_i^{\prime}\} \hspace{0.5cm}(\{i,j,k,l\}=\{1,2,3,4\}).\]
We call $\left(P_1^{\prime}, P_2^{\prime}, P_3^{\prime}, P_4^{\prime}\right)$ the {\em dual} {\em (quadruplet)} of $\left(P_1, P_2, P_3, P_4\right)$. 

Let us denote a quadruplet and its dual by $\mathcal{Q}$ and $\mathcal{Q}^{\prime}$ in what follows. 
\end{definition}

It follows directly from the definition that the points of the dual quadruplet are cospherical with those of the original one. 

Since the dual quadruplet can be defined by circles and angles, 
\begin{lemma}\label{lem-conf_inv}
The notion of the dual quadruplet is conformally invariant. 
Namely, $\left(T(\mathcal{Q})\right)^{\prime}=T\left(\mathcal{Q}^{\prime}\right)$ for a M\"obius transformation $T$. 
% the dual of the image of a quadruplet by a M\"obius transformation is equal to the image of the dual of the original quadruplet by the same M\"obius transformation. 
%a M\"obius transformation maps the dual of a quadruplet $\{P_1, P_2, P_3, P_4\}$ to the dual of $\{T(P_1), T(P_2), T(P_3), T(P_4)\}$. 
%if $\{P_1^{\prime}, P_2^{\prime}, P_3^{\prime}, P_4^{\prime}\}$ is the dual of $\{P_1, P_2, P_3, P_4\}$ then a M\"obius transformation $T$ maps 
%$\{T(P_1^{\prime}), T(P_2^{\prime}), T(P_3^{\prime}), T(P_4^{\prime})\}$ is the dual of 
\end{lemma}
It allows us to assume that $P_1, P_2, P_3$, and $P_4$ are four points in $\mathbb{R}^3$ (or $\mathbb{R}^3\cup\{\infty\}$) which are not cocircular nor collinear. 

\begin{example} \rm 
(1) When the four points are vertices of a regular tetrahedron in $\mathbb{R}^3$, the dual can be obtained by the symmetry in the barycenter. 

(2) The dual of 
\[(1,0,0), \,(0,1,0), \,(0,-1,0), \>\mbox{ and }\>(0,0,1)\] 
are 
\[\left(-\frac1{\sqrt2}, 0, \frac1{\sqrt2}\right), \,\left(\frac12, -\frac1{\sqrt2}, \frac12\right), \,\left(\frac12, \frac1{\sqrt2}, \frac12\right),\>\mbox{ and }\>\left(\frac1{\sqrt2}, 0, -\frac1{\sqrt2}\right).\] 
\end{example}

Our main theorem is: 
\begin{theorem}\label{thm} 
The dual of a dual quadruplet coincides with the original one: $\mathcal{Q}^{\prime\prime}=\mathcal{Q}$. 
\end{theorem}

\subsection{Proof of Theorem \ref{thm}}
Let $\left(P_1^{\prime\prime}, P_2^{\prime\prime}, P_3^{\prime\prime}, P_4^{\prime\prime}\right)$ be the dual of $\left(P_1^{\prime}, P_2^{\prime}, P_3^{\prime}, P_4^{\prime}\right)$. 
We work in the extended plane $\pi_1(\varSigma)\cong\varPi_1\cup\{\infty\}$ in what follows. 
We have only to show $\pi_1\left(P_1^{\prime\prime}\right)=\infty$. 

\smallskip
Put 
\[\widetilde P_j=\pi_1(P_j) \hspace{0.4cm}\mbox{and}\hspace{0.4cm}\widetilde P_j^{\prime}=\pi_1\big(P_j^{\prime}\big)\]
for $j=1,2,3,4$. We remark $\widetilde P_1=\infty$. 

Let $\widehat P_j$ $(j=2,3,4)$ be the excenters of the triangle $\triangle \widetilde P_2\widetilde P_3\widetilde P_4$ (Figure \ref{conf_dual2-orthocenter}). 
\begin{figure}[htbp]
\begin{center}
\includegraphics[width=.5\linewidth]{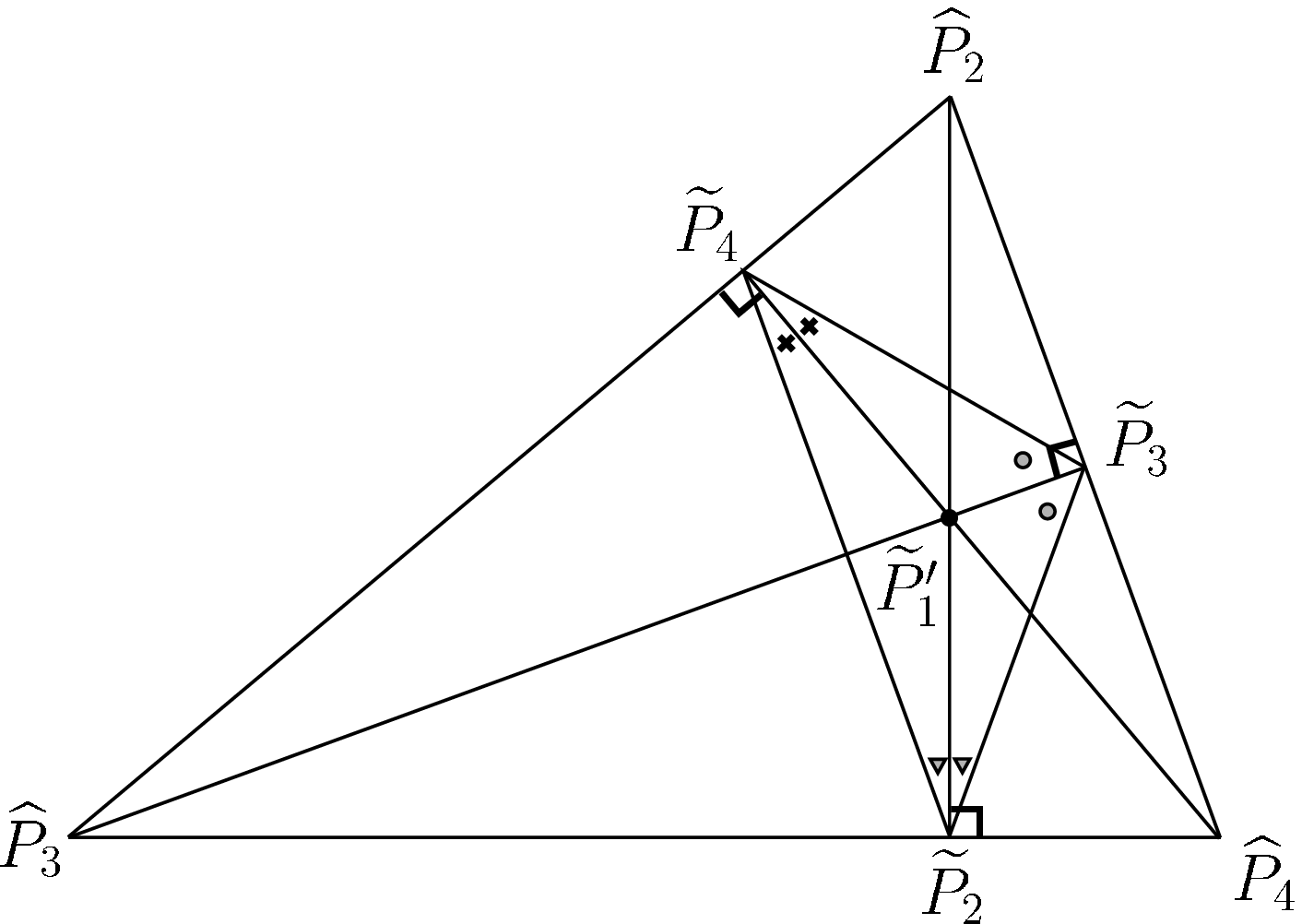}
\caption{The incenter and excenters of $\triangle \widetilde P_2\widetilde P_3\widetilde P_4$}
\label{conf_dual2-orthocenter}
\end{center}
\end{figure}
\begin{lemma}\label{lem2} 
We have $\widetilde P_j^{\prime}=\widehat P_j$ for $j=2,3,4$. 
%Namely, through a stereographic projection $\pi_1$, the dual quadruplet is mapped to the union of the incenter and the excenters of the triangle $\triangle \widetilde P_2\widetilde P_3\widetilde P_4$. 
Namely, through a stereographic projection from a point of a quadruplet $\mathcal{Q}$, the dual can be obtained as the union of the incenter and the excenters of the triangle whose vertices are the images of the other three points of $\mathcal{Q}$. 
\end{lemma}
\begin{proof}
We show $\widetilde P_2^{\prime}=\widehat P_2$. 
%Suppose $\{j,k,l\}=\{2,3,4\}$. 
%As as $\widehat P_j, \widetilde P_k, \widetilde P_l$, and $\widetilde P_1^{\prime}$ are cocircular and $\angle \widetilde P_1^{\prime}\widetilde P_j\widetilde P_k=\angle \widetilde P_1^{\prime}\widetilde P_j\widetilde P_l$ (Figure \ref{conf_dual2'}), the power of a point theorem implies 
%\begin{equation}\label{f_houbeki}
%\big|\widetilde P_j\widetilde P_k\big|\big|\widetilde P_j\widetilde P_l\big|
%=\big|\widetilde P_j\widetilde P_1^{\prime}\big|\big|\widetilde P_j\widehat P_j\big|.
%\end{equation}
%
As as $\widehat P_2, \widetilde P_3, \widetilde P_4$, and $\widetilde P_1^{\prime}$ are cocircular and $\angle \widetilde P_1^{\prime}\widetilde P_2\widetilde P_3=\angle \widetilde P_1^{\prime}\widetilde P_2\widetilde P_4$ (Figure \ref{conf_dual2'}), the power of a point theorem implies 
\begin{equation}\label{f_houbeki}
\big|\widetilde P_2\widetilde P_3\big|\big|\widetilde P_2\widetilde P_4\big|
=\big|\widetilde P_2\widetilde P_1^{\prime}\big|\big|\widetilde P_2\widehat P_2\big|.
\end{equation}
\begin{figure}[htbp]
\begin{center}
\includegraphics[width=.5\linewidth]{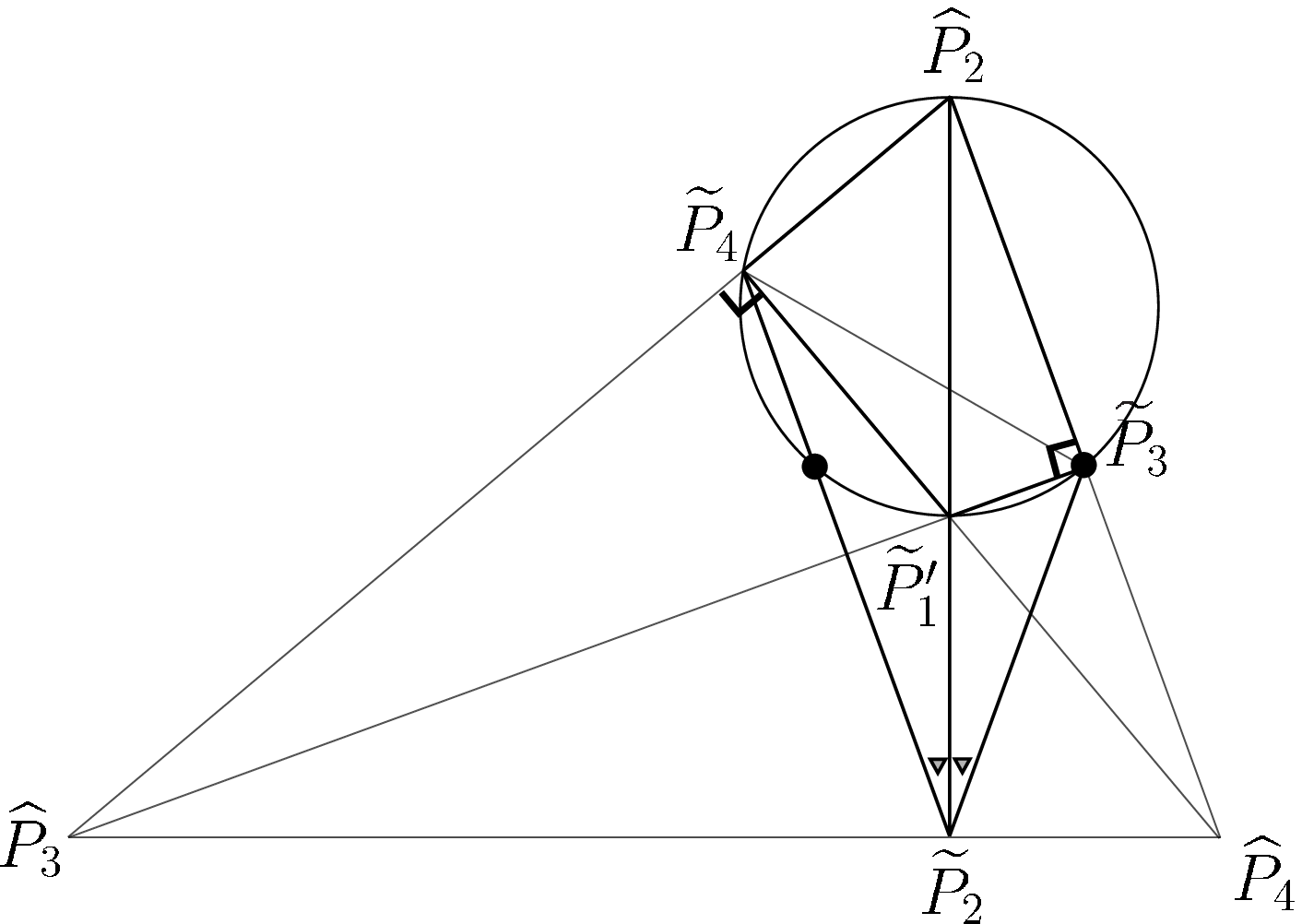}
\caption{$\big|\widetilde P_2\widetilde P_3\big|\big|\widetilde P_2\widetilde P_4\big|
=\big|\widetilde P_2\widetilde P_1^{\prime}\big|\big|\widetilde P_2\widehat P_j\big|$}
\label{conf_dual2'}
\end{center}
\end{figure}

Consider an inversion $I_2$ in a circle with center $\widetilde P_2$ and radius $r=\sqrt{\big|\widetilde P_2\widetilde P_3\big|\big|\widetilde P_2\widetilde P_4\big|}\,$. 
Then \setlength\arraycolsep{1pt}
\[\begin{array}{rclrcl}
I_2(\widetilde P_1)&=&\widetilde P_2, & &&\\[1mm]
\overline{\widetilde P_2I_2\big(\widetilde P_3\big)}&=&\overline{\widetilde P_2\widetilde P_3}, & \hspace{0.5cm}
\big|\widetilde P_2I_2\big(\widetilde P_3\big)\big|&=&\big|\widetilde P_2\widetilde P_4\big|, \\[2mm]
\overline{\widetilde P_2I_2\big(\widetilde P_4\big)}&=&\overline{\widetilde P_2\widetilde P_4}, & \hspace{0.5cm}
\big|\widetilde P_2I_2\big(\widetilde P_4\big)\big|&=&\big|\widetilde P_2\widetilde P_3\big|,
\end{array}\]
which implies that the triangle $\triangle I_2(\widetilde P_1)I_2\big(\widetilde P_3\big)I_2\big(\widetilde P_4\big)$ is symmetric with the triangle $\triangle \widetilde P_2\widetilde P_3\widetilde P_4$ in the line $\overline{\widetilde P_2\widetilde P_1^{\prime}}$ (Figure \ref{conf_dual3-I2} right). 
\begin{figure}[htbp]
\begin{center}
\includegraphics[width=.8\linewidth]{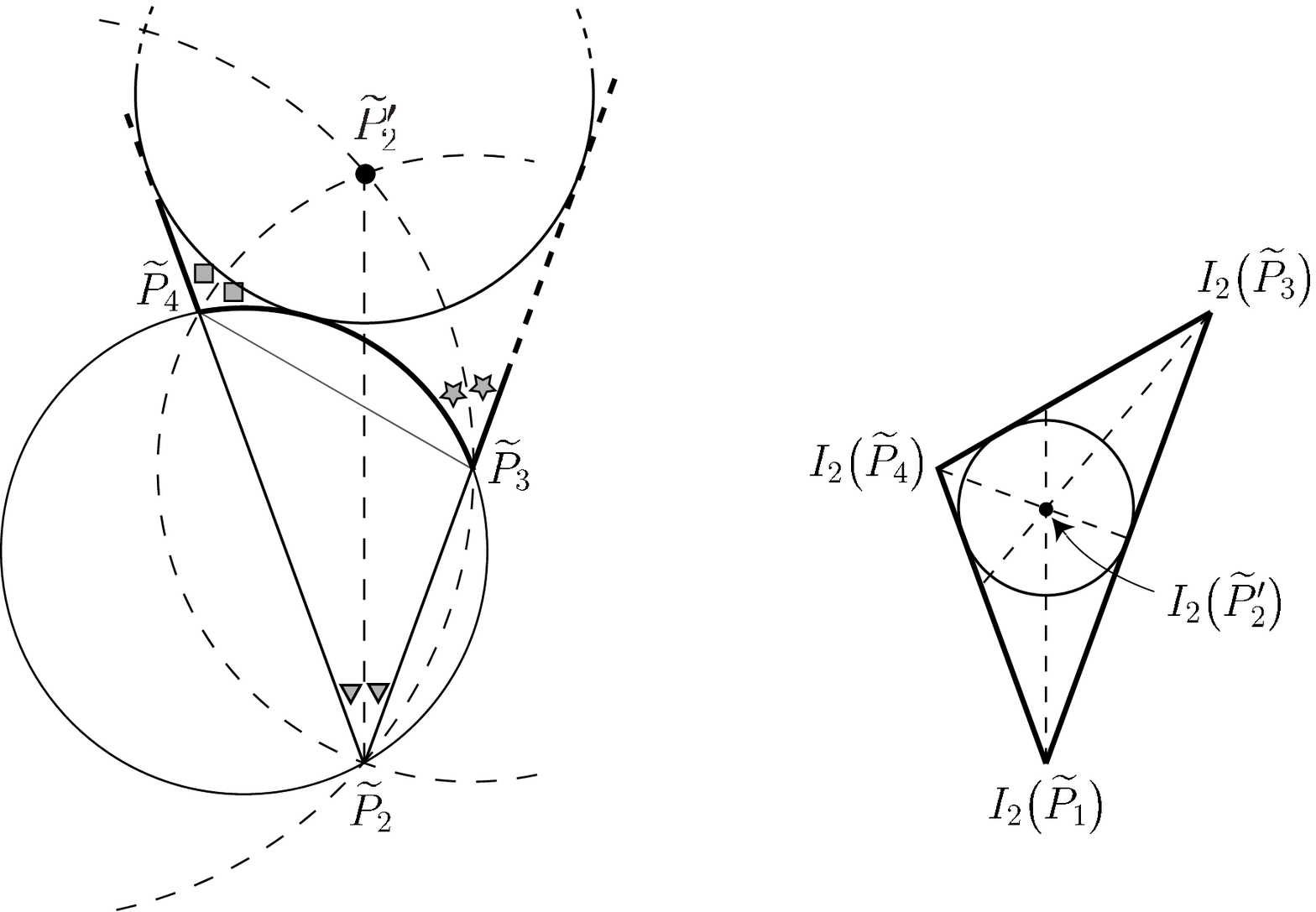}
\caption{$\widetilde P_2^{\prime}$ (left) and the image by the inversion $I_2$ (right)}
\label{conf_dual3-I2}
\end{center}
\end{figure}
Since $I_2\big(\widetilde P_2^{\prime}\big)$ is the incenter of the triangle $\triangle I_2(\widetilde P_1)I_2\big(\widetilde P_3\big)I_2\big(\widetilde P_4\big)$, we have $I_2\big(\widetilde P_2^{\prime}\big)=\widetilde P_1^{\prime}$. 

On the other hand, the formula (\ref{f_houbeki}) implies that $I_2\big(\widehat P_2\big)=\widetilde P_1^{\prime}$. 
Therefore $\widetilde P_2^{\prime}=\widehat P_2$. 
\end{proof}
\begin{corollary}
The dual quadruplet consists of points which are not cocircular (nor collinear when points are in $\mathbb{R}^3$). 
Therefore we can take its dual once again. 
\end{corollary}
\begin{lemma}\label{lem4-equiradius}
Three circles through $\widetilde P_1^{\prime}$, $\varGamma\big(\widetilde P_1^{\prime}, \widetilde P_i^{\prime}, \widetilde P_j^{\prime}\big)$ $(\{i,j\}\subset\{2,3,4\})$, have the same radius {\rm (Figure \ref{conf_dual5-equiradius})}. 
%
%Three circles through $\widetilde P_1^{\prime}$, 
%\[\varGamma\big(\widetilde P_1^{\prime}, \widetilde P_2^{\prime}, \widetilde P_3^{\prime}\big), \varGamma\big(\widetilde P_1^{\prime}, \widetilde P_2^{\prime}, \widetilde P_4^{\prime}\big), \>\mbox{ and }\>\varGamma\big(\widetilde P_1^{\prime}, \widetilde P_3^{\prime}, \widetilde P_4^{\prime}\big),\]
%have the same radius. 
%Three circles through $\widetilde P_1^{\prime}$, $\varGamma\big(\widetilde P_1^{\prime}, \widetilde P_2^{\prime}, \widetilde P_3^{\prime}\big), \,\varGamma\big(\widetilde P_1^{\prime}, \widetilde P_2^{\prime}, \widetilde P_4^{\prime}\big)$, and $\varGamma\big(\widetilde P_1^{\prime}, \widetilde P_3^{\prime}, \widetilde P_4^{\prime}\big)$, have the same radius. 
\end{lemma}
\begin{proof}
First observe that, as $\widetilde P_1^{\prime}, \widetilde P_j^{\prime}, \widetilde P_k$, and $\widetilde P_l$ $(\{j,k,l\}=\{2,3,4\})$ are cocircular (Figure \ref{conf_dual4'-cocircular}), we have 
\begin{figure}[htbp]
\begin{center}
\includegraphics[width=.6\linewidth]{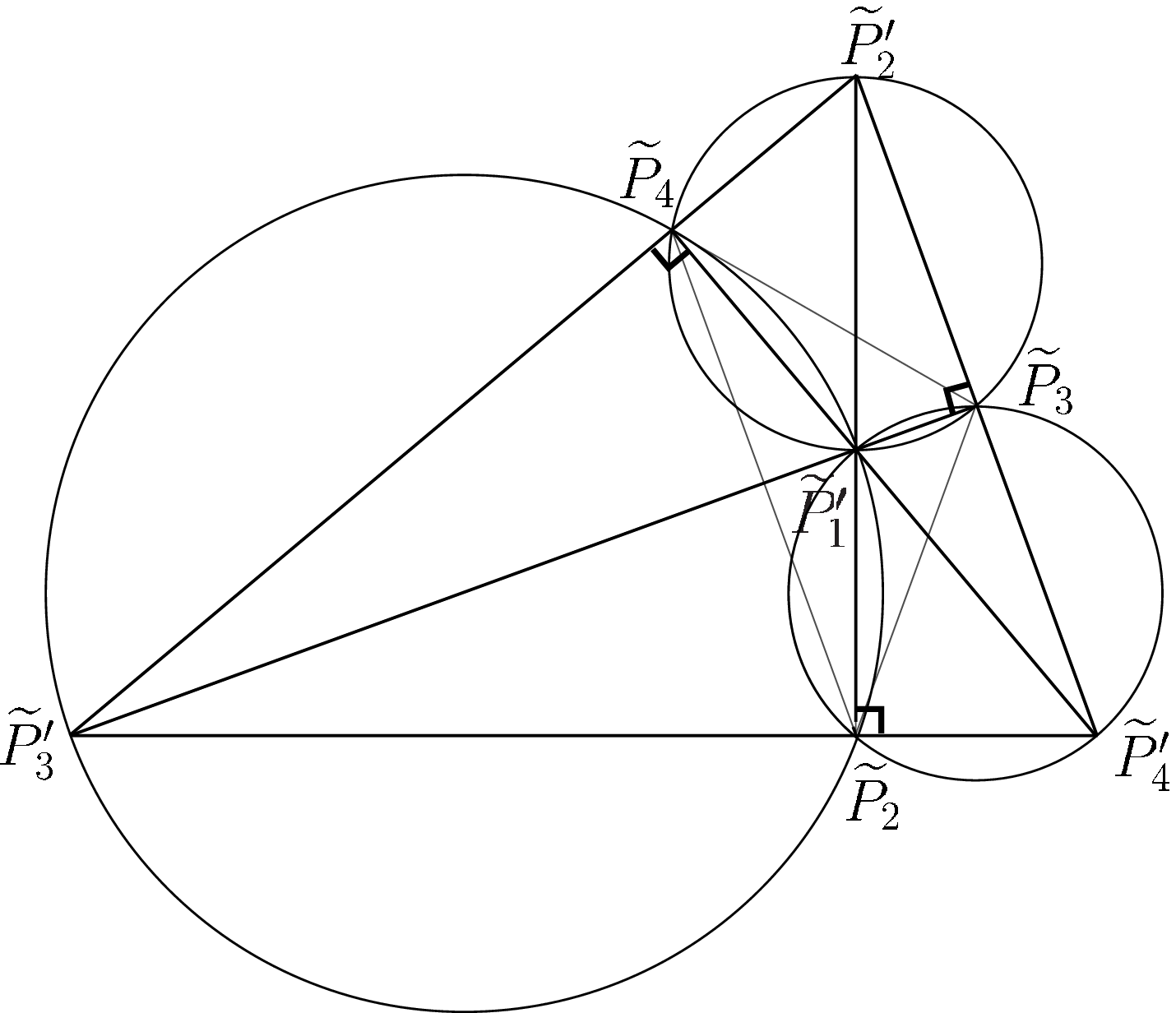}
\caption{$\widetilde P_1^{\prime}, \widetilde P_j^{\prime}, \widetilde P_k$, and $\widetilde P_l$ are cocircular}
\label{conf_dual4'-cocircular}
\end{center}
\end{figure}

\begin{equation}\label{f_angles}
\begin{array}{l}
\angle \widetilde P_1^{\prime}\widetilde P_2\widetilde P_3=\angle \widetilde P_1^{\prime}\widetilde P_2\widetilde P_4=\angle \widetilde P_1^{\prime}\widetilde P_3^{\prime}\widetilde P_2^{\prime}=\angle \widetilde P_1^{\prime}\widetilde P_4^{\prime}\widetilde P_2^{\prime},\\[1mm]
\angle \widetilde P_1^{\prime}\widetilde P_3\widetilde P_2=\angle \widetilde P_1^{\prime}\widetilde P_3\widetilde P_4=\angle \widetilde P_1^{\prime}\widetilde P_2^{\prime}\widetilde P_3^{\prime}=\angle \widetilde P_1^{\prime}\widetilde P_4^{\prime}\widetilde P_3^{\prime},\\[1mm]
\angle \widetilde P_1^{\prime}\widetilde P_4\widetilde P_2=\angle \widetilde P_1^{\prime}\widetilde P_4\widetilde P_3=\angle \widetilde P_1^{\prime}\widetilde P_2^{\prime}\widetilde P_4^{\prime}=\angle \widetilde P_1^{\prime}\widetilde P_3^{\prime}\widetilde P_4^{\prime}
\end{array}
\end{equation}
(Figure \ref{conf_dual4-angles}). 
\begin{figure}[htbp]
\begin{center}
\includegraphics[width=.55\linewidth]{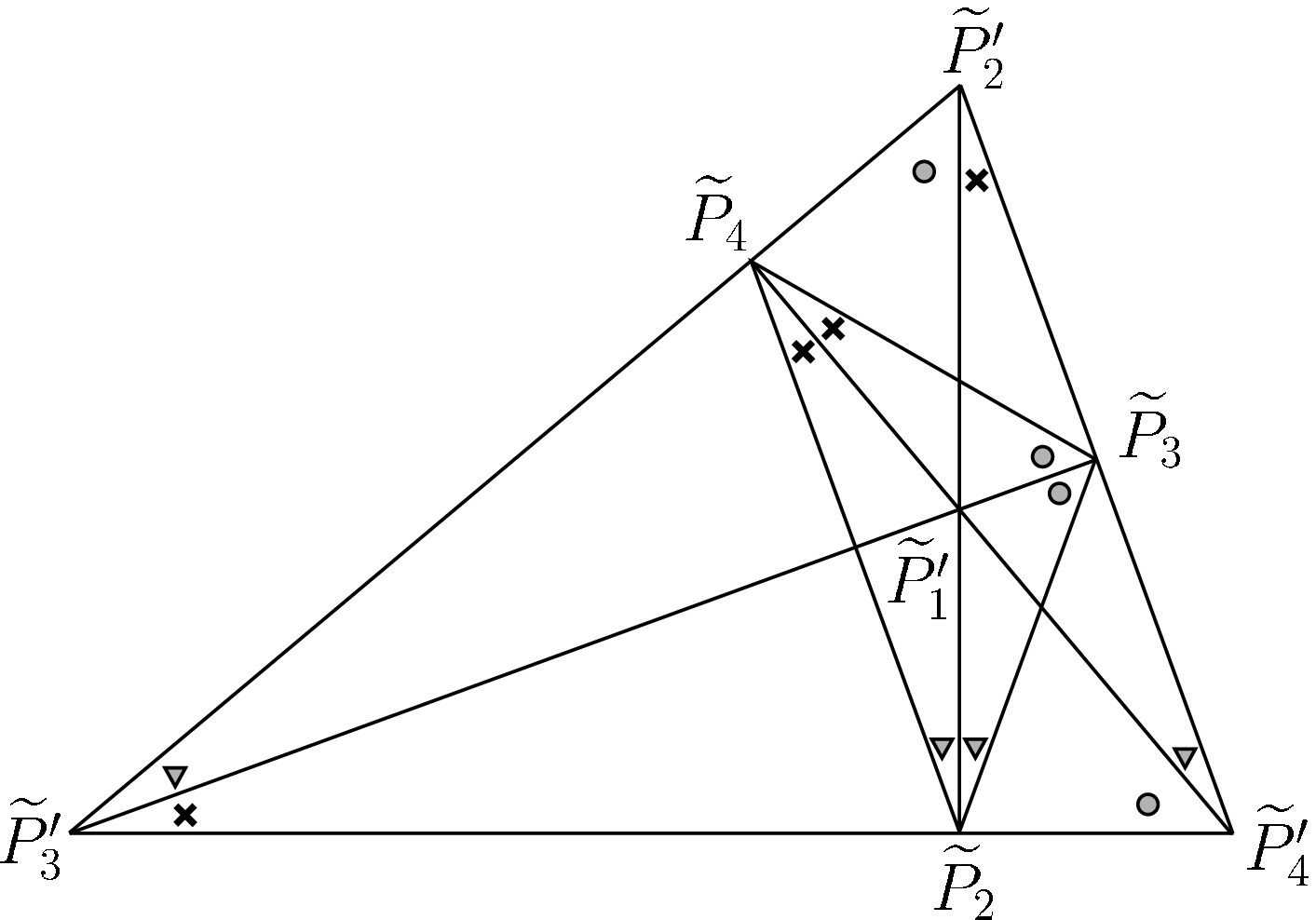}
\caption{The equalities between angles}%$\widetilde P_i, \widetilde P_j, \widetilde P_i^{\prime}$, and $\widetilde P_j^{\prime}$ $(\{i,j\}\subset\{2,3,4\})$ are cocircular
\label{conf_dual4-angles}
\end{center}
\end{figure}
The equality $\angle \widetilde P_1^{\prime}\widetilde P_2^{\prime}\widetilde P_3^{\prime}=\angle \widetilde P_1^{\prime}\widetilde P_4^{\prime}\widetilde P_3^{\prime}$ means that the inscribed angles for the chord $\widetilde P_1^{\prime}\widetilde P_3^{\prime}$ of two circles $\varGamma\big(\widetilde P_1^{\prime}, \widetilde P_2^{\prime}, \widetilde P_3^{\prime}\big)$ and $\varGamma\big(\widetilde P_1^{\prime}, \widetilde P_4^{\prime}, \widetilde P_3
^{\prime}\big)$ coincide (Figure \ref{conf_dual5-equiradius}), which implies that the radii of the two circles are equal. 
\begin{figure}[htbp]
\begin{center}
\includegraphics[width=.8\linewidth]{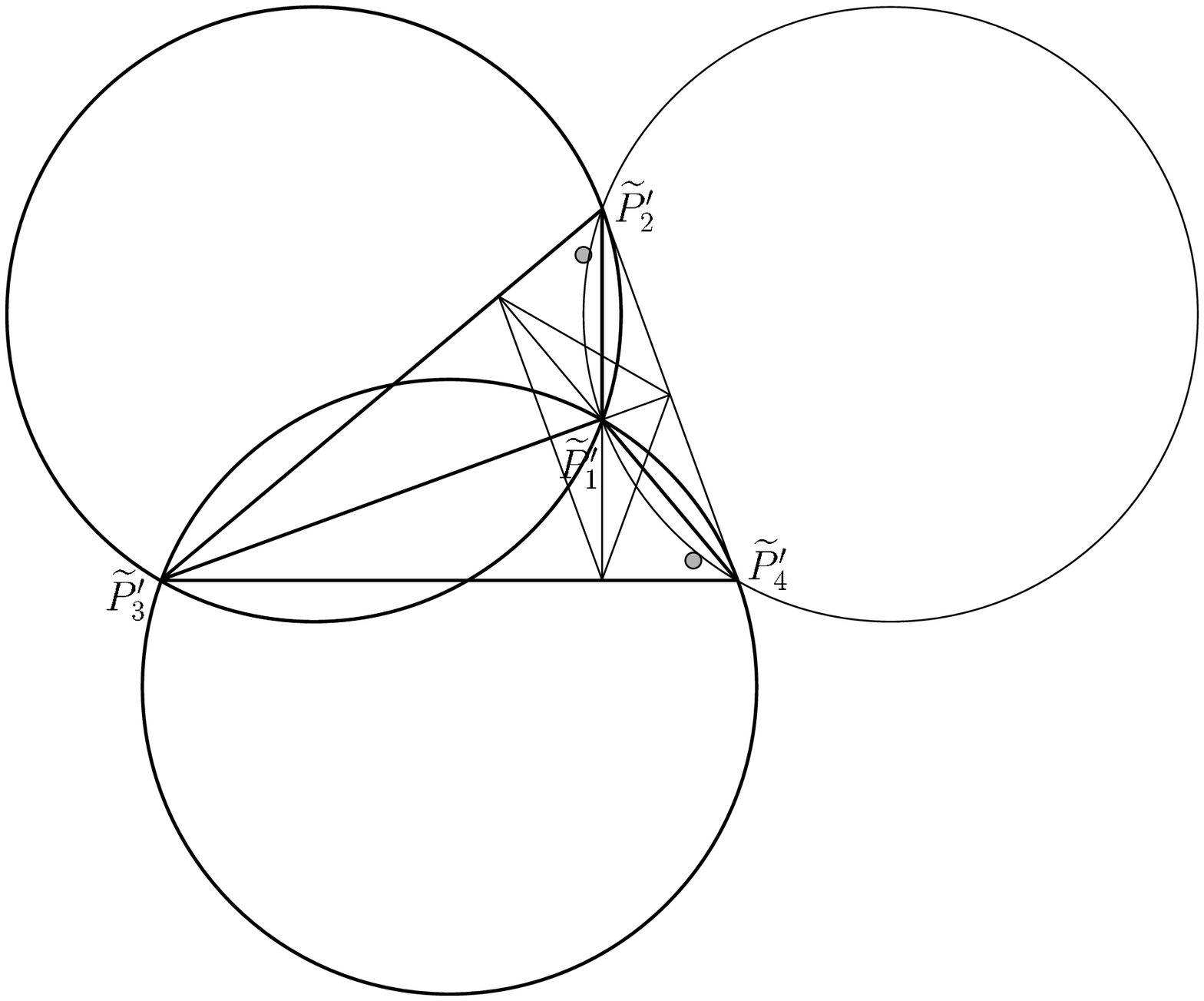}
\caption{$\angle \widetilde P_1^{\prime}\widetilde P_2^{\prime}\widetilde P_3^{\prime}=\angle \widetilde P_1^{\prime}\widetilde P_4^{\prime}\widetilde P_3^{\prime}\Longrightarrow r\big(\varGamma\big(\widetilde P_1^{\prime}, \widetilde P_2^{\prime}, \widetilde P_3^{\prime}\big)\big)=r\big(\varGamma\big(\widetilde P_1^{\prime}, \widetilde P_4^{\prime}, \widetilde P_3^{\prime}\big)\big)$ }
\label{conf_dual5-equiradius}
\end{center}
\end{figure}
\end{proof}
\begin{proofoftheorem} {\bf \ref{thm}}. 
Suppose $\mathcal{Q}^{\prime\prime}=\left(P_1^{\prime\prime}, P_2^{\prime\prime}, P_3^{\prime\prime}, P_4^{\prime\prime}\right)$ is the dual of $\mathcal{Q}^{\prime}=\left(P_1^{\prime}, P_2^{\prime}, P_3^{\prime}, P_4^{\prime}\right)$, which is the dual of $\mathcal{Q}=(P_1, P_2, P_3, P_4)$. 
We have only to show $\pi_1\big(P_1^{\prime\prime}\big)=\infty$. 

By definition, $\pi_1\big(P_1^{\prime\prime}\big)$ belongs to the intersection of the three circular angle bisectors of pairs of circles among $\varGamma\big(\widetilde P_1^{\prime}, \widetilde P_i^{\prime}, \widetilde P_j^{\prime}\big)$ $(\{i,j\}\subset\{2,3,4\})$. 
Since these three circles have the same radius by Lemma \ref{lem4-equiradius}, any circular angle bisector of a pair among them is a straight line (Figure \ref{conf_dual6-circ_ang_bisec}). 
\begin{figure}[htbp]
\begin{center}
\includegraphics[width=.9\linewidth]{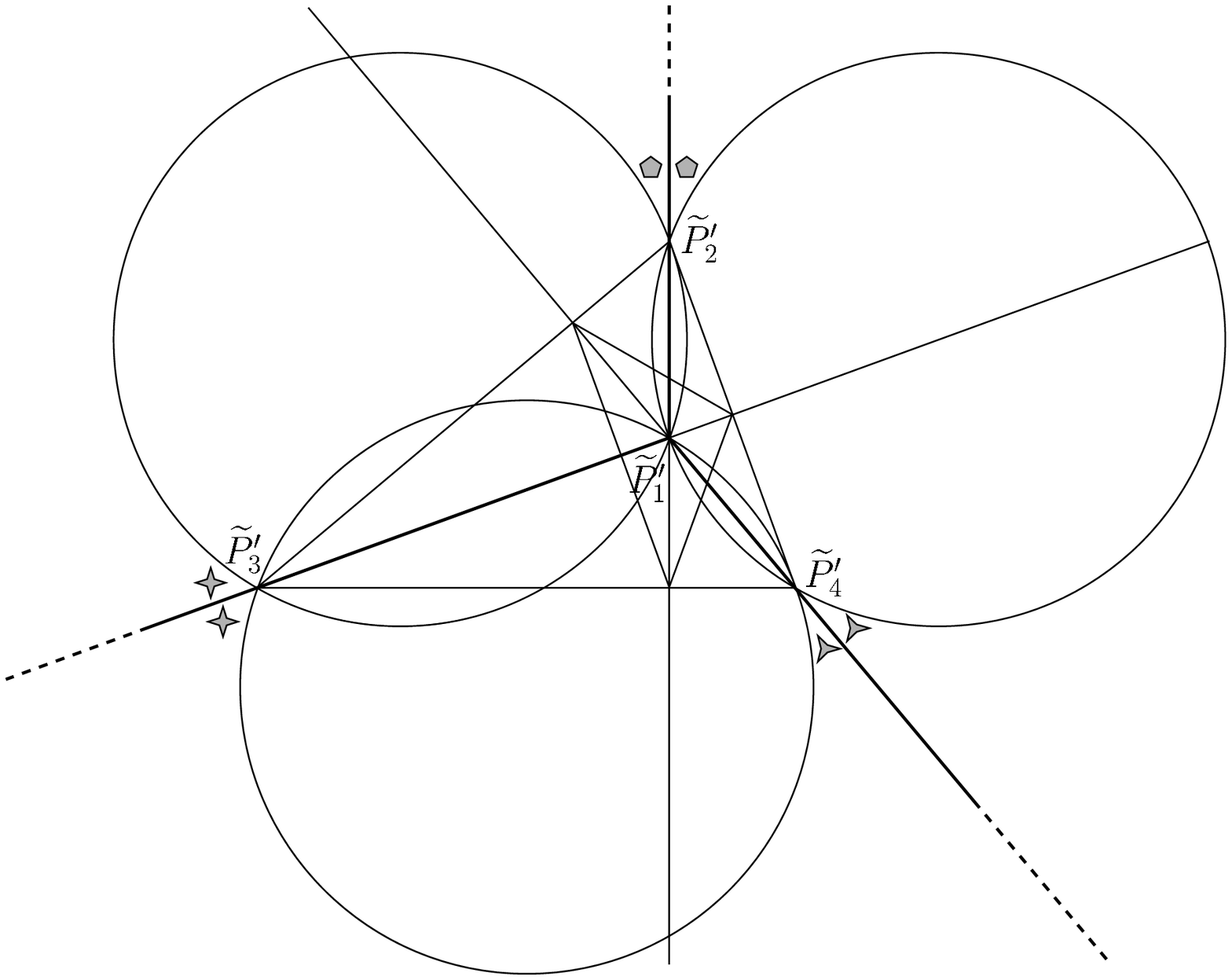}
\caption{The three circular angle bisectors}
\label{conf_dual6-circ_ang_bisec}
\end{center}
\end{figure}
Therefore, the three circular angle bisectors meet in $\widetilde P_1^{\prime}$ and $\infty$, which completes the proof. 
\end{proofoftheorem}
We remark we can also show $\pi_1\big(P_2^{\prime\prime}\big)=\widetilde P_2$. 
Figure \ref{conf_dual4'-cocircular} implies 
\begin{equation}\label{f_houbeki2}
\big|\widetilde P_2^{\prime}\widetilde P_3^{\prime}\big|\big|\widetilde P_2^{\prime}\widetilde P_4\big|
=\big|\widetilde P_2^{\prime}\widetilde P_2\big|\big|\widetilde P_2^{\prime}\widetilde P_1^{\prime}\big|
=\big|\widetilde P_2^{\prime}\widetilde P_4^{\prime}\big|\big|\widetilde P_2^{\prime}\widetilde P_3\big|. 
\end{equation}
Let $I^{\prime}$ be an inversion in a circle with center $\widetilde P_2^{\prime}$ with radius the square root of (\ref{f_houbeki2}). 
Then the formula (\ref{f_houbeki2}) implies that $I^{\prime}\big(\widetilde P_3^{\prime}\big)=\widetilde P_4, I^{\prime}\big(\widetilde P_4^{\prime}\big)=\widetilde P_3$, and $I^{\prime}\big(\widetilde P_1^{\prime}\big)=\widetilde P_2$. 
As $I^{\prime}\big(\pi_1\big(P_2^{\prime\prime}\big)\big)$ is the incenter of the triangle $\triangle I^{\prime}\big(\widetilde P_3^{\prime}\big)I^{\prime}\big(\widetilde P_4^{\prime}\big)I^{\prime}\big(\widetilde P_1^{\prime}\big)$, we have $I^{\prime}\big(\pi_1\big(P_2^{\prime\prime}\big)\big)=\widetilde P_1^{\prime}$, which implies $\pi_1\big(P_2^{\prime\prime}\big)=I^{\prime}\big(\widetilde P_1^{\prime}\big)=\widetilde P_2$. 

\subsection{Cross ratio of the dual quadruplet}
Suppose $\left(P_1^{\prime}, P_2^{\prime}, P_3^{\prime}, P_4^{\prime}\right)$ is the dual of $(P_1, P_2, P_3, P_4)$, and $\varSigma$ is a sphere through $P_i$ and $P_j^{\prime}$ as before. 
Let $p$ be a stereographic projection from $\varSigma$ to $\mathbb{C}\cup\{\infty\}$. 
We identify $P_i$ (or $P_j^{\prime}$) with the complex number $p(P_i)$ (or respectively $p(P_j^{\prime})$). 
\begin{theorem}\label{thm2}
The cross ratio of the dual quadruplet 
\[\mbox{\sl cr}^{\prime}=\frac{P_2^{\prime}-P_1^{\prime}}{P_2^{\prime}-P_4^{\prime}}\cdot\frac{P_3^{\prime}-P_4^{\prime}}{P_3^{\prime}-P_1^{\prime}}\]
is equal to the complex conjugate of the cross ratio of the original quadruplet 
\[\mbox{\sl cr}=\frac{P_2-P_1}{P_2-P_4}\cdot\frac{P_3-P_4}{P_3-P_1}\,.\]
\end{theorem}
\begin{proof} 
First observe that the cross ratio is independent of a stereographic projection as far as it is orientation preserving. 
We may assum without loss of generality that $p=\pi_1$, $\widetilde P_1=\infty$, $\widetilde P_4=0$, $\widetilde P_2=1$, and $\widetilde P_3=\mbox{\sl cr}$. 

\smallskip
(i) Figure \ref{conf_dual4-angles} implies 
\[\arg(\mbox{\sl cr}^{\prime})=-\angle \widetilde P_1^{\prime}\widetilde P_2^{\prime}\widetilde P_4^{\prime}-\angle \widetilde P_4^{\prime}\widetilde P_3^{\prime}\widetilde P_1^{\prime}
=-\angle \widetilde P_2\widetilde P_4\widetilde P_3
=-\arg(\mbox{\sl cr}).\]

(ii) Since $\widetilde P_1^{\prime}, \widetilde P_2, \widetilde P_4^{\prime}$, and $\widetilde P_3$ are cocircular, we have 
\[\frac{\big|\widetilde P_2^{\prime}-\widetilde P_1^{\prime}\big|}{\big|\widetilde P_2^{\prime}-\widetilde P_4^{\prime}\big|}
=\frac{\big|\widetilde P_2^{\prime}-\widetilde P_3\big|}{\big|\widetilde P_2^{\prime}-\widetilde P_2\big|} \>\>\mbox{ and }\>\>
\frac{\big|\widetilde P_3^{\prime}-\widetilde P_4^{\prime}\big|}{\big|\widetilde P_3^{\prime}-\widetilde P_1^{\prime}\big|}
=\frac{\big|\widetilde P_3^{\prime}-\widetilde P_3\big|}{\big|\widetilde P_3^{\prime}-\widetilde P_2\big|}\,. \]
Since $\angle \widetilde P_2^{\prime}\widetilde P_3\widetilde P_3^{\prime}=\angle \widetilde P_2^{\prime}\widetilde P_2\widetilde P_3^{\prime}=\frac{\pi}2$ we have 
\[|\mbox{\sl cr}^{\prime}|=\frac{\big|\widetilde P_2^{\prime}-\widetilde P_1^{\prime}\big|}{\big|\widetilde P_2^{\prime}-\widetilde P_4^{\prime}\big|}\cdot\frac{\big|\widetilde P_3^{\prime}-\widetilde P_4^{\prime}\big|}{\big|\widetilde P_3^{\prime}-\widetilde P_1^{\prime}\big|}
=\frac{\big|\widetilde P_2^{\prime}-\widetilde P_3\big|\cdot \big|\widetilde P_3^{\prime}-\widetilde P_3\big|}{\big|\widetilde P_2^{\prime}-\widetilde P_2\big|\cdot \big|\widetilde P_3^{\prime}-\widetilde P_2\big|}
=\frac{\big|\triangle\widetilde P_2^{\prime}\widetilde P_3^{\prime}\widetilde P_3\big|}{\big|\triangle\widetilde P_2^{\prime}\widetilde P_3^{\prime}\widetilde P_2\big|}\,. \]
As $\angle \widetilde P_3\widetilde P_4\widetilde P_2^{\prime}=\angle \widetilde P_2\widetilde P_4\widetilde P_3^{\prime}$ the right hand side of the above is equal to 
\[\frac{\big|\widetilde P_4-\widetilde P_3\big|}{\big|\widetilde P_4-\widetilde P_2\big|}=|\mbox{\sl cr}|,\]
which completes the proof. 
\end{proof}

\section{The cocircular case}\label{sec_cocircular}
We define the conformal dual of a quadruplet of cocircular points $P_i$ by the limit of the conformal dual of a quadruplet of non-cocircular points $X_i$ as $X_i$ approach $P_i$. 

Suppose $\widetilde P_4=0, \widetilde P_2=1, \widetilde P_3=c$, and $\widetilde P_1=\infty$ $(0<c<1)$ after a stereographic projection from $\varSigma$ to $\mathbb C\cup\{\infty\}$. 
Let $\{\widetilde X_n\}$ be a sequence of points in $\mathbb C\setminus\mathbb R$ which approach $c$ as $n$ goes to $\infty$. 
Figure \ref{conf_dual_cocirc} implies that the dual of $\left(\infty, 0, \widetilde X_n, 1\right)$ tends to be $\left(c, 1, \infty, 0\right)$ as $n$ goes to $\infty$. 
\begin{figure}[htbp]
\begin{center}
\includegraphics[width=.9\linewidth]{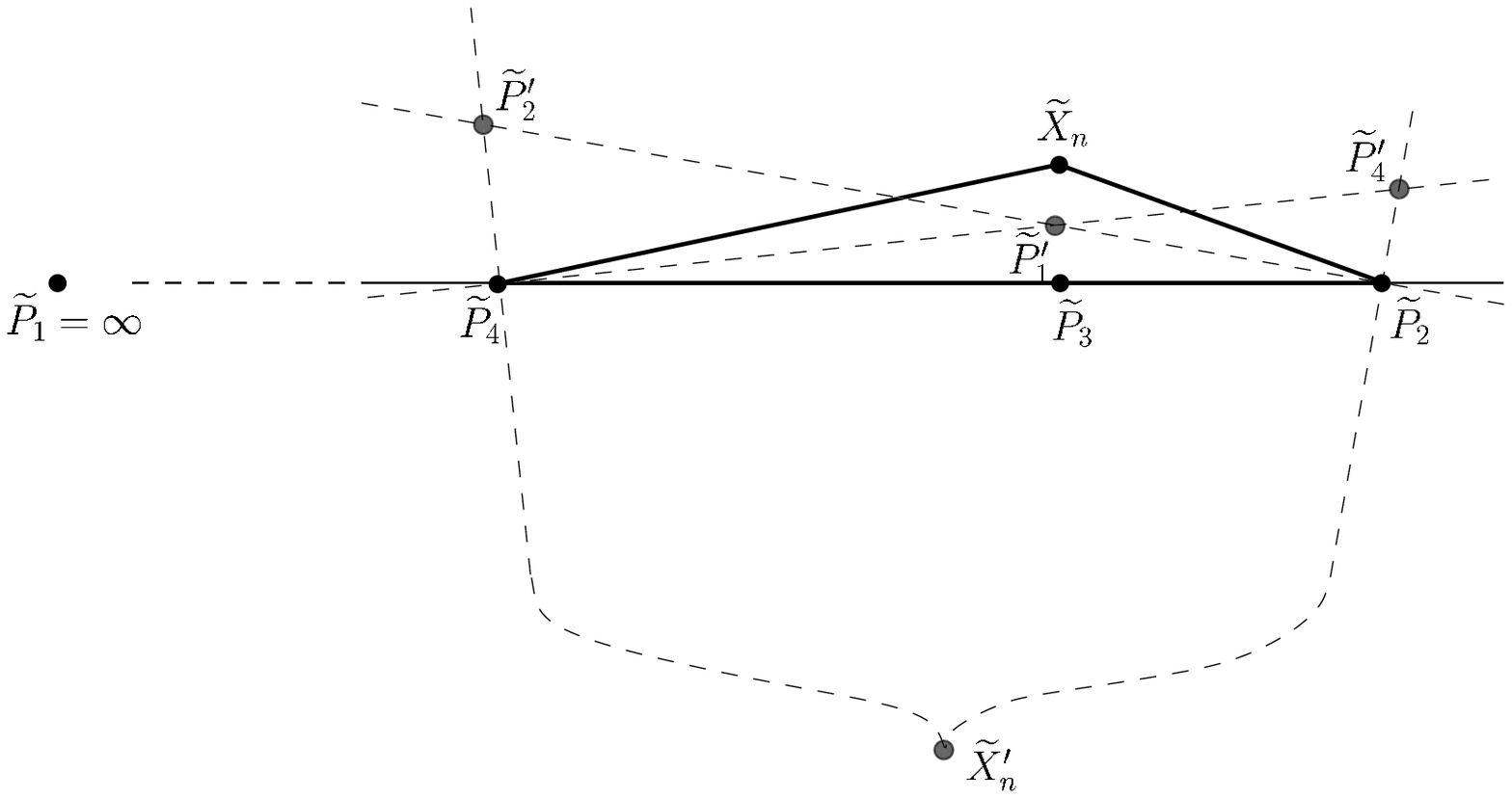}
\caption{The four dotted lines are angle bisectors}
\label{conf_dual_cocirc}
\end{center}
\end{figure}

Thus we are lead to 
\begin{definition} \rm 
Suppose $P_1, P_2, P_3$, and $P_4$ are points on a circle in this order. 
Then the dual of $(P_1, P_2, P_3, P_4)$ is given by $(P_3, P_4, P_1, P_2)$. 
Namely, diagonal points are being exchanged. 
\end{definition}

It is obvious that both Theorem \ref{thm} and Theorem \ref{thm2} also hold for cocircular cases.

%\begin{thebibliography}{E} 

%\end{thebibliography}

\bigskip \noindent

\bigskip \noindent
Department of Mathematics, Tokyo Metropolitan University, \\
1-1 Minami-Ohsawa, Hachiouji-Shi, Tokyo 192-0397, JAPAN. \\
E-mail: ohara@tmu.ac.jp
\end{document}